\documentclass[final,3p,times]{elsarticle}
\usepackage{graphicx}
\usepackage{amssymb}
\usepackage{amsthm}
\newtheorem{prop}{Proposition}
\usepackage{amsmath}
\usepackage{color}

\renewcommand{\div}{\operatorname{div}}
\newcommand{\grad}{\boldsymbol{\nabla}}
\newcommand{\pd}[2]{\frac{\partial#1}{\partial#2}}

\def\D{\mathcal{D}}

\def\dt{\Delta t}
\def\R{\mathbb{R}}
\def\U{\mathrm{W}}  
\def\Uz{\mathrm{V}}

\def\Q{\mathrm{Q}}
\def\Qz{\mathring{\mathrm{Q}}}
\def\Qh{\mathrm{Q}_h}
\def\f{\boldsymbol{f}}

\def\bu{\boldsymbol{u}}
\def\bv{\boldsymbol{v}}
\def\bw{\boldsymbol{w}}
\def\gbu{\grad\bu}
\def\gbv{\grad\bv}

\def\mynu{\mu}

\begin{document}

\begin{frontmatter}

\title{A posteriori analysis of Chorin-Temam scheme for Stokes equations \\
Analyse a posteriori du sch\'ema Chorin-Temam pour les \'equations de Stokes} 

\author[ENPC]{S\'ebastien Boyaval}
\author[EPFL]{Marco Picasso}
\address[ENPC]{
Universit\'e Paris Est, Laboratoire d'hydraulique Saint-Venant (Ecole Nationale des Ponts et Chauss\'ees -- EDF R\&D -- CETMEF), 
78401 Chatou Cedex, France ; and INRIA, MICMAC team--project, Rocquencourt, France (sebastien.boyaval@enpc.fr). Corresponding author.
}
\address[EPFL]{
MATHICSE, Station 8, Ecole Polytechnique F\'ed\'erale de Lausanne, 1015 Lausanne, Switzerland (marco.picasso@epfl.ch). 
}

\begin{abstract}
We consider Chorin-Temam scheme (the simplest pressure-correction projection method)  for the time-discretization of an 
unstationary Stokes problem 
in $\D\subset\R^d$ ($d=2,3$) given $\mynu,\f$:\ {$\sf (P)$} find $(\bu,p)$ solution to $\bu|_{t=0}=\bu_0$, $\bu|_{\partial\D}=0$ and
\begin{equation}
\label{eq:unstationarystokes}
\pd{\bu}{t} - \mynu \Delta \bu + \grad p = \f \qquad \div \bu = 0 \qquad \text{ on } (0,T)\times\D\,.
\end{equation}
Inspired by the analyses of the Backward Euler scheme performed by C.Bernardi and R.Verf\"urth, 
we derive a posteriori estimators for the error on $\grad\bu$ in $L^2(0,T;L^2(\D))$-norm. Our invesigation is supported by numerical experiments.
 \vskip 0.5\baselineskip \noindent
{\it French version:} On discr\'etise en temps par le sch\'ema Chorin-Temam un probl\`eme de Stokes non-stationnaire pos\'e dans $\D\subset\R^d$ ($d=2,3$) \'etant donn\'es $\mynu,\f$:\
{$\sf (P)$} trouver $(\bu,p)$ solution de $\bu|_{t=0}=\bu_0$, $\bu|_{\partial\D}=0$ et~\eqref{eq:unstationarystokes}. En s'insipirant des analyses de C.Bernardi and R.Verf\"urth
pour le sch\'ema Euler r\'etrograde, nous construisons des estimateurs a posteriori pour l'erreur commise sur $\grad\bu$ en norme $L^2(0,T;L^2(\D))$. 
Notre \'etude est \'etay\'ee par des exp\'eriences num\'eriques.
\end{abstract}


\end{frontmatter}

\section*{Version fran\c{c}aise abr\'eg\'ee}

\'Etant donn\'es $\mynu>0$, $\f\in L^2(0,T;{\Q}^d)$ 
et $\bu_0\in \Uz$ ($\Q$ et $\Uz$ sont d\'efinis en~\eqref{espace} ci-dessous), on discr\'etise en temps la formulation faible
{$\rm (P)$} du probl\`eme {$\sf (P)$} (\'equation 
\eqref{eq:stokesQvar} ci-dessous) par le sch\'ema Chorin-Temam:
soit $\bu^{-1/2}=\bu_0$, $p^0=0$, pour $n=0\ldots N-1$, 
\'etant donn\'es $\dt^n\in(0,\dt]$ et $\f^{n+1} = \frac1{\dt^{n}} \int_{t^{n}}^{t^{n+1}} \f(s)ds $\; (o\`u { $t_n = \sum_{k=0}^{n-1}\dt^k$; $t_N=T$}), 
{$(\mathrm{P}^n)$:} on cherche $\bu^{n+1/2}\in\U$, $p^{n+1} \in \Qz\cap H^1(\D)$ solutions de~(\ref{eq:chorin-temamPoisson1C}--\ref{eq:chorin-temamPoisson2C}).
La convergence a priori vers des solutions de {$\rm (P)$} quand $\dt\to0$ et son ordre sont connus~\cite{guermond-1996,prohl-1997,guermond-quartapelle-1998} (voir Prop.~\ref{prop1}).
Mais on aimerait ici estimer {\it a posteriori} l'erreur de discr\'etisation en temps pour bien choisir les pas 
$\dt^n$ en pratique, ce qui est encore un probl\`eme ouvert.
Les estimateurs d'erreur a posteriori propos\'es dans~\cite{bernardi-verfuerth-2004,verfuerth-2010} pour la semi-discr\'etisation en temps avec le sch\'ema Euler r\'etrograde ne sont pas valables ici.
Et les r\'ecentes analyses~\cite{kharrat-mghazli-2010,kharrat-mghazli-2011} pour la semi-discr\'etisation en temps avec le sch\'ema Chorin-Temam proposent un estimateur (diff\'erent des n\^otres) qui ne tient pas compte de tous les termes d'erreur.

Apr\`es avoir d\'efini les r\'esidus~(\ref{eq:residualRupoisson}--\ref{eq:residualRppoisson}) d'une approximation Chorin-Temam $\bu^{\dt}$, $p^{\dt}$ de la solution du probl\`eme $\rm (P)$ construite comme en~\eqref{reconstruct}, nous suivons dans ce travail la proc\'edure g\'en\'erique d'analyse a posteriori des \'equations de Stokes instationnaires qui est pr\'esent\'ee dans~\cite{verfuerth-2010}.
Nous testons donc l'\'equation~\eqref{eq:error_ct} v\'erifi\'ee par $e_u=\bu-\bu^{\dt}$, $e_p = p- p^{\dt}$ avec $\bv=e_u-\Pi e_u\in\Uz$, $q=0$, o\`u $\Pi$ est un op\'erateur de projection dans $\U$ qui conserve la divergence. 
De~\eqref{eq:error_ct_100} dans $\D'(0,T)$, on tire alors l'in\'egalit\'e~\eqref{eq:error_ct_104} avec~\eqref{eq:constants}, $\div e_u= R_p$,~\eqref{eq:H1stability}, $\Pi e_u=-\Pi\bu^{\dt}$, $e_u(0)=0$, et une int\'egration par parties. 
On obtient ensuite une borne sup\'erieure~\eqref{eq:up} en utilisant par exemple~\eqref{sharp1},~\eqref{sharp2} selon~\cite{verfuerth-2010}. 
D'autre part, on obtient aussi la borne inf\'erieure~\eqref{eq:down} 
si en plus de~\eqref{firststep} (obtenue facilement avec~\eqref{eq:residualRupoisson},~\eqref{eq:residualRppoisson} et $\div\bu=0$) on utilise~\eqref{new}.
La proc\'edure d'analyse a posteriori de~\cite{verfuerth-2010} permet donc bien d'obtenir des bornes sup\'erieures et inf\'erieures compl\`etement calculables de l'erreur $ \|\grad e_u\|_{L^2(0,T;\Q^{d\times d})}^2 +\|\partial_t e_u + \grad e_p\|_{L^2(0,T;\U')}^2$ (voir Prop.~\ref{above} et Prop.~\ref{below}).
Mais d'une part, il vaut mieux utiliser $\|\div\partial_t\bu^{\dt}\|_{L^2(0,t;{\Q}^d)}$ (plut\^ot que $\|\div\partial_t\bu^{\dt}\|_{L^1(0,t;{\Q}^d)}$ si on suit strictement~\cite{verfuerth-2010}),
donc l'estimateur~\eqref{estimator2} plut\^ot que~\eqref{estimator1} (tir\'e directement de Prop.~\ref{above} et Prop.~\ref{below}) si on veut une estimation robuste (c'est-\`a-dire de qualit\'e ind\'ependante des param\`etres de discr\'etisation).
D'autre part, bien que notre estimation ne soit pas totalement efficace 
(comme dans~\cite{kharrat-mghazli-2010,kharrat-mghazli-2011}, nos estimateurs ne sont pas born\'es inf\'erieurement {\it et} sup\'erieurement par l'erreur),
on montre n\'eanmoins num\'eriquement qu'elle peut \^etre utile dans certains cas, et en particulier qu'elle est plus pr\'ecise que celle propos\'ee dans~\cite{kharrat-mghazli-2010,kharrat-mghazli-2011} (plus de termes d'erreur sont pris en compte).

Pour des approximations (\`a $\lambda>0$ donn\'e) des composantes du vecteur vitesse et de la pression
$$ \bu=\pi\sin(\lambda t)\left(\sin(2\pi y)\sin(\pi x)^2;-\sin(2\pi x)\sin(\pi y)^2\right) \qquad p=\sin(\lambda t)\cos(\pi x)\sin(\pi y) $$
avec des \'el\'ements finis continus $\mathbb{P}_2$ et $\mathbb{P}_1$ par morceaux dans $\D\equiv(-1,1)\times(-1,1)$ ($d=2$) maill\'e r\'eguli\`erement avec des simplexes, 
on a calcul\'e num\'eriquement l'efficacit\'e  des estimateurs~\eqref{estimator1},~\eqref{estimator2} et~\eqref{estimator3} 
pour $t\in(0,T)$ discr\'etis\'e avec des pas de temps constants $\dt=T/N$ ($N\in\mathbb{N}$).
En effet, 
notre analyse a posteriori du cas semi-discret en temps se prolonge au cas compl\`etement discret
(en d\'ecomposant les r\'esidus discrets 
en composantes temporelles et spatiales comme dans~\cite{verfuerth-2010} 
on obtient directement les versions discr\`etes en espace des estimateurs semi-discrets en temps plus des estimateurs pour l'erreur en espace),
et l'erreur de discr\'etisation en espace est par ailleurs n\'egligeable ici pour notre exemple num\'erique (comme observ\'e dans~\cite{guermond-minev-shen-2006} o\`u il est utilis\'e pour $\lambda=1$).
Pour $\lambda=10$, l'estimateur~\eqref{estimator2} est meilleur que~\eqref{estimator1} (qui n'est pas robuste si $T$ est grand ou $\dt$ petit) et~\eqref{estimator3} (dont  l'efficacit\'e diminue avec $\dt$ car des termes d'erreur sont omis, alors qu'ils sont bien pris en compte par~\eqref{estimator2}).
Toutefois, notre estimateur~\eqref{estimator2} ne repr\'esente pas toujours bien l'erreur lui non plus, m\^eme si on lui ajoute le terme $\|\div\bu^{\dt}\|_{L^\infty(0,t;{\Q}^d)}^2$ de la borne sup\'erieure~\eqref{eq:up} (a priori pas born\'e sup\'erieurement par l'erreur~\eqref{error}).
Pour $\lambda=1$ par exemple, l'erreur d\'ecro\^it avec $\dt$ comme $\|\div\bu^{\dt}\|_{L^\infty(0,t;{\Q}^d)}^2$, mais ce terme est d'un ordre de grandeur bien inf\'erieur aux autres termes de~\eqref{estimator2} (ou~\eqref{estimator3}) donc on ne le voit que pour $\dt$ assez petit m\^eme si on ajoute le terme $\|\div\bu^{\dt}\|_{L^\infty(0,t;{\Q}^d)}^2$ \`a l'estimateur~\eqref{estimator2}.
Sans parler de l'estimation de l'erreur sur $\bu$ en norme $L^\infty(0,T;L^2(\D))$, on n'a donc pas encore totalement r\'esolu le probl\`eme de trouver un estimateur efficace et robuste pour l'erreur commise sur $\grad\bu$ en norme $L^2(0,T;L^2(\D))$ par le sch\'ema Chorin-Temam.
Il faudrait au moins ajouter des coefficients devant les termes de l'estimateur~\eqref{estimator2} plus $\|\div\bu^{\dt}\|_{L^\infty(0,t;{\Q}^d)}^2$ si on veut l'utiliser en pratique.
N\'eanmoins, nous esp\'erons que cette \'etude apporte un nouvel \'eclairage \`a la question.

\section{Numerical solutions to Stokes equations by Chorin-Temam pressure-correction projection method}

Given a smooth bounded open set $\D\subset\R^d$ ($d=2,3$)
with boundary $\partial\D$ of class $C^2$, let us denote similarly by $(\cdot,\cdot)$ the usual $L^2$ inner-products for scalar and vector functions in $\D$
and introduce the standard functional spaces~\cite{temam-1979,girault-raviart-1979} 
\begin{equation}
\label{espace}
\Q  := L^2(\D) \,, \quad
\Qz := 
\{q\in L^2(\D)\,,\,\int_\D q=0\} \,, \quad 
\U  := [H^{1}_0(\D)]^d \,, \quad
\Uz := 
 \{\bv\in[H^{1}_0(\D)]^d\,,\,\div\bv=0\} 
\,.
\end{equation}
We consider a weak 
formulation of problem {$\sf (P)$} with $\mynu>0$, $\f\in L^2(0,T;{\Q}^d)$ 
(given in a Bochner space), $\bu_0\in \Uz$: \\
{$\rm (P)$} find $\bu\in L^2(0,T;\U)$ and $p\in L^2(0,T;\Qz)$ such that $\bu(0)=\bu_0$ in $\Uz$, and the following equation holds in $L^2(0,T)$
\begin{equation}
\label{eq:stokesQ}
\frac{d}{dt} ( \bu,\bv ) + \mynu (\gbu, \gbv) -(p,\div\bv) + (q,\div\bu) = (\f,\bv) \,,\ \forall (\bv,q)\in\U\times\Q \,.
\end{equation}
It is well-known that 
problem $\rm (P)$ is well-posed~\cite{temam-1979,girault-raviart-1979}
(in particular, $\bu\in C([0,T],\Uz)$ so initial condition makes sense)
and because of the regularity assumptions, it also holds $\partial_t\bu\in L^2((0,T)\times\D)$, $p\in L^2(0,T;H^1(\D))$ and in $L^2(0,T)$
\begin{equation}
\label{eq:stokesQvar}
( \partial_t\bu,\bv ) + \mynu (\gbu, \gbv) + (\grad p,\bv) - (\grad q,\bu) = (\f,\bv) \,,\ \forall (\bv,q)\in\U\times H^1(\D)\,.
\end{equation}
A standard time-discretization of~\eqref{eq:stokesQvar} is Chorin-Temam scheme~\cite{chorin-1968,temam-1968}: given $\bu^{-1/2}=\bu_0$, $p^0=0$, for $n=0\ldots N-1$, 
given $\dt^n\in(0,\dt]$,
$\f^{\dt} = \frac1{\dt^{n}} \int_{t^{n}}^{t^{n+1}} \f(s)ds $\; 
({\footnotesize 
$t_n = \sum_{k=0}^{n-1}\dt^k$; 
$t_N=T$}),\;
{$(\mathrm{P}^n)$} find 
$\bu^{n+1/2}\in\U$, 
$p^{n+1} \in \Qz\cap H^1(\D)$ 
solutions to
\begin{subequations}
\begin{equation}
\label{eq:chorin-temamPoisson1C}
  (\frac{\bu^{n+1/2}-\bu^{n-1/2}}{\dt^n} + \grad p^{n},\bv) + \mynu(\gbu^{n+1/2},\gbv) = (\f^{n+1},\bv) \quad \forall \bv \in \Uz \,,
\end{equation}
\begin{equation}
\label{eq:chorin-temamPoisson2C}
 \frac1{\dt^{n+1}} (\div\bu^{n+1/2}, q) = - (\grad p^{n+1},\grad q) \quad  \forall q \in \Q \,,
\end{equation}
\end{subequations}
which yields approximations whose rate of convergence to solutions of {$\rm (P)$} is well-known {\it  a priori}
\cite{guermond-1996,prohl-1997,guermond-quartapelle-1998}: 
\begin{prop}
\label{prop1}
The following estimate holds:
\begin{equation}
\label{apriori}
  \| \bu^{\dt}-\bu \|_{L^2(0,T;\U)} + \| p^{\dt}-p \|_{L^2(0,T;\Q)} = O(\dt^\frac12) \text{ as $\dt\to0$ }\,,
\end{equation}
where $\bu^{\dt}$ and $p^{\dt}$ are defined as
\begin{equation} 
\label{reconstruct}
\bu^{\dt}(t) = \frac{t-t_{n}}{\dt^n} \bu^{n+1/2} - \frac{t-t_{n+1}}{\dt^n} \bu^{n-1/2} \,,\ 
\quad
p^{\dt}(t) = p^{n} \,.\ 
\quad 
\forall t\in(t_n,t_{n+1}]
\,,
\end{equation} 
\end{prop}

In this work, we would like to numerically evaluate {\it a posteriori} the time discretization error
with a view to adequately choosing the time steps $\dt^n$ of Chorin-Temam scheme in practice (under a given error tolerance),
which is still an open problem.
A posteriori error estimations 
have been proposed for the Backward-Euler scheme (including full discretizations, in time {\it and space})~\cite{bernardi-verfuerth-2004,verfuerth-2010} 
but they do not straightforwardly apply here.
And a posteriori analyses of 
Chorin-Temam scheme have indeed been carried out recently~\cite{kharrat-mghazli-2010,kharrat-mghazli-2011}
but they suggest an estimator (different than ours) that does not account for the whole error.
The present invesigation focuses on fully computable error bounds for Chorin-Temam scheme derived from the generic a posteriori framework introduced in~\cite{verfuerth-2010} for the unstationary Stokes equations.
Although our estimator is a priori not fully efficient, it is better than other ones and useful in some cases.

Note that in the following, we 
denote by $a\lesssim b$ any relation $a \leq C b$ between two real numbers $a,b$ where $C>0$ 
is a numerical constant independent of the data of the problem.
Moreover, we shall use standard inequalities such as 
$$
\|\div \bv\|_{\Q} \le d^{1/2} \|\grad \bv\|_{\Q^{d\times d}}\,,\ \forall\bv\in\U
$$
and Poincar\'e-Friedrichs inequality with constant $C_P(\D)>0$, then also
\begin{equation}
\label{eq:constants}
\max(\|\bv\|_{{\Q}^d}^2,\|\grad\bv\|_{\Q^{d\times d}}^2) \le \|\bv\|_{\U}^2 
\le (1+C_P^2) \|\grad \bv\|_{\Q^{d\times d}}^2 \,,\ \forall \bv\in\U \,.
\end{equation}
In Section~\ref{apost}, we derive a posteriori error estimates following the procedure of~\cite{verfuerth-2010}, i.e. 
invoking $\Pi:\U\to\U$, 
a projection  
such that $\bv-\Pi\bv\in\Uz$. For all $\bv\in\U$, $\Pi\bv$ is 
the solution of Stokes equations: $\exists!\, q_{\bv}\in\Qz,\exists \Upsilon(\D)>0$ such that
\begin{subequations}
\begin{equation}
\label{eq:stokesprojection}
(\grad\Pi\bv,\grad\bw) = (q_{\bv},\div\bw)\quad (r,\div\Pi\bv) = (r,\div\bv)\,, \ \forall (\bw,r)\in\U\times \Qz \,,
\end{equation}
\begin{equation}
\label{eq:H1stability}
\Upsilon \|\grad\Pi\bv\|_{\Q^{d\times d}}\le\|\div\bv\|_{\Q} 
\,.
\end{equation}
\end{subequations}
In Section~\ref{num}, we numerically test our a posteriori estimator.

\section{A posteriori estimation of semi-discrete errors}
\label{apost}

Let us define residuals for Chorin-Temam approximations $\bu^{\dt}$, $p^{\dt}$ as in~\eqref{reconstruct} of the solution to the problem $\rm (P)$
\begin{subequations}
\begin{multline}
\label{eq:residualRupoisson}
<R_u,\bv>_{\U',\U} 
= (\f,\bv) - (\partial_t\bu^{\dt},\bv) - (\grad p^{\dt},\bv) - \mynu (\grad\bu^{\dt}, \gbv) 
\equiv (\f-\f^{\dt},\bv) + \mynu(\grad\bu^{\dt,+}-\grad\bu^{\dt}, \gbv) \,,\ \forall \bv\in\U \,, 
\end{multline}
\begin{equation}
\label{eq:residualRppoisson}
(R_p,q)  = -(\div \bu^{\dt},q) \,,\ \forall q\in\Q \,,
\end{equation}
\end{subequations}
where $\f^{\dt}=\f^{n+1}\,,\ \bu^{\dt,+}=\bu^{n+\frac12}$ for $t\in(t_n,t_{n+1}]$. The 
errors $ e_u=\bu-\bu^{\dt}$, $e_p = p- p^{\dt}$ satisfy:
\begin{equation}
\label{eq:error_ct}
(\partial_t e_u + \grad e_p,\bv) + \mynu(\grad e_u, \gbv) + (\div e_u,q) = < R_u , \bv >_{\U',\U} + (R_p,q) \,,\ \forall (\bv,q) \in \U\times\Q\ \,.
\end{equation}
Testing~\eqref{eq:error_ct} against $\bv=e_u-\Pi e_u\in\Uz$, $q=0$, yields in $\D'(0,T)$ (distributional sense)
\begin{equation}
\label{eq:error_ct_100}
\frac12 \frac{d}{dt} \| e_u\|_{{\Q}^d}^2 + \mynu\|\grad  e_u\|_{\Q^{d\times d}}^2 = <R_u, e_u>_{\U',\U} - <R_u,\Pi e_u>_{\U',\U} 
+ (\partial_t e_u, \Pi e_u) + \mynu(\grad  e_u ,\grad\Pi e_u) \,.
\end{equation}
Using Young inequality with~\eqref{eq:constants}, 
$\div e_u 
= R_p$, 
~\eqref{eq:H1stability}, 
$\Pi e_u=-\Pi\bu^{\dt}$, 
$e_u(0)=0$, 
and integrating by part, 
one obtains 
\begin{equation}
\label{eq:error_ct_104}
\| e_u\|_{L^{\infty}(0,t;{\Q}^d)}^2 + \mynu \|\grad  e_u\|_{L^{2}(0,t;\Q^{d\times d})}^2
\lesssim 
\|R_u\|_{L^{2}(0,t;\U')}^2 + \mynu \|R_p\|_{L^{2}(0,t;\Q)}^2 
+ \int_0^t \left| ( e_u, \partial_t\Pi e_u ) \right| + \|(e_u,\Pi e_u)\|_{L^\infty(0,t)} 
\,.
\end{equation}

If we follow~\cite{verfuerth-2010}, then~\eqref{eq:error_ct_104} yields a computable upper-bound using the following inequalities with Young's one
\begin{subequations}
\begin{equation}
\label{sharp1}
\int_0^t \left| ( e_u, \partial_t\Pi\bu^{\dt}) \right| \le \| e_u\|_{L^{\infty}(0,t;{\Q}^d)} \|\Pi\partial_t\bu^{\dt}\|_{L^{1}(0,t;{\Q}^d)}
\end{equation}
\begin{equation}
\label{sharp2}
\|(e_u,\Pi\bu^{\dt})\|_{L^\infty(0,t)}
\lesssim 
\| e_u\|_{L^{\infty}(0,t;{\Q}^d)} \|\div\bu^{\dt}\|_{L^{\infty}(0,t;{\Q}^d)}
\,.
\end{equation}
\end{subequations}
Since\footnote{
 Observe that the convergence of $\partial_t \bu^{\dt} + \grad p^{\dt}$ to $\partial_t \bu + \grad p$ in $L^2(0,T;\U')$ is natural here, like for Backward-Euler schemes~\cite{bernardi-verfuerth-2004}. 
}
$ \|\partial_t e_u + \grad e_p\|_{L^2(0,T;\U')}^2 \le 2 \|R_u\|_{L^2(0,T;\U')}^2 + 2 \|\grad e_u\|_{L^2(0,T;\Q^{d\times d})}^2$
also holds from~\eqref{eq:residualRupoisson}, one indeed obtains from~\eqref{eq:H1stability}:
\begin{prop}\label{above}
There exists a constant $c^+(\D)>0$ such that the following computable estimations 
hold
\begin{multline}
\label{eq:up}
\frac1{c^+} \max\left( \| e_u\|_{L^{\infty}(0,T;{\Q}^d)}^2, \|\partial_t e_u + \grad e_p\|_{L^2(0,T;\U')}^2, \mynu\|\grad e_u\|_{L^2(0,T;\Q^{d\times d})}^2 \right) \\
 \le  
\|\f-\f^{\dt}\|_{L^{2}(0,T;\Q^d)}^2 
+ 
\mynu\|\grad\bu^{\dt,+}-\grad\bu^{\dt}\|_{L^{2}(0,T;\Q^{d\times d})}^2  
+
\mynu\|\div \bu^{\dt}\|_{L^{2}(0,T;{\Q})}^2 
+
\|\div\partial_t\bu^{\dt}\|_{L^{1}(0,T;{\Q})}^2 
+ \|\div \bu^{\dt}\|_{L^{\infty}(0,T;{\Q})}^2 
\,.
\end{multline}
\end{prop}
On the other hand, from~\eqref{eq:residualRupoisson},~\eqref{eq:residualRppoisson} and $\div\bu=0$, one has 
\begin{multline}
\label{firststep}
  \mynu\|\grad\bu^{\dt,+}-\grad\bu^{\dt}\|_{L^2(0,T;\Q^{d\times d})}^2 
+ \mynu\|\div \bu^{\dt}\|_{L^{2}(0,T;{\Q})}^2 
\lesssim 
  \|\f-\f^{\dt}\|_{L^2(0,T;\Q^d)}^2 
+ \|\partial_t e_u + \grad e_p\|_{L^2(0,T;\U')}^2 
+ \mynu\|\grad e_u\|_{L^2(0,T;\Q^{d\times d})}^2 \,,
\end{multline}
from which one next straightforwardly obtains the counterpart of~\eqref{eq:up} if one uses, in addition to~\eqref{firststep},
\begin{equation}
\label{new}
 \|\div\partial_t\bu^{\dt}\|_{L^{1}(0,T;{\Q})}^2 \lesssim \frac{T}{\min_{n=0\ldots N-1}|\dt^n|^2} \|\grad e_u\|_{L^{2}(0,T;{\Q}^{d\times d})}^2 \,.
\end{equation}
\begin{prop}\label{below}
There exists a constant $c^-(\D)>0$ such that the following computable lower bound holds
\begin{multline}
\label{eq:down}
c^- \left( 
\mynu\|\grad\bu^{\dt,+}-\grad\bu^{\dt}\|_{L^2(0,T;\Q^{d\times d})}^2 
+ 
\mynu\|\div \bu^{\dt}\|_{L^{2}(0,T;{\Q})}^2 
+ 
\frac{1}{N} \|\div\partial_t\bu^{\dt}\|_{L^{1}(0,T;{\Q})}^2
\right)
\\
\le 
\|\f-\f^{\dt}\|_{L^2(0,T;\Q^d)}^2 + \|\partial_t e_u + \grad e_p\|_{L^2(0,T;\U')}^2 
+ \left(\mynu+\frac1{\min_{n=0\ldots N-1}|\dt^n|}\right) \|\grad e_u\|_{L^2(0,T;\Q^{d\times d})}^2
\,.
\end{multline}
\end{prop}
\begin{proof}[Proof of~\eqref{new}.] We use the following inequality with $\div\bu=0$, noting
$\frac6{\dt^n}\int_{t^n}^{t^{n+1}} \|\div\bu^{\dt}\|_{\Q}^2\ge\|\div\bu^{n+\frac12}\|_{\Q}^2+\|\div\bu^{n-\frac12}\|_{\Q}^2$:\\
$ \displaystyle
\|\div\partial_t\bu^{\dt}\|_{L^{1}(0,T;{\Q})}^2 \le N \sum_{n=0}^{N-1} \|\div(\bu^{n+\frac12}-\bu^{n-\frac12})\|_{\Q}^2 
\le 2N \sum_{n=0}^{N-1} (\|\div\bu^{n+\frac12}\|_{\Q}^2+\|\div\bu^{n-\frac12}\|_{\Q}^2) \le \sum_{n=0}^{N-1} \frac{12N}{\dt^n}\int_{t^n}^{t^{n+1}} \|\div\bu^{\dt}\|_{\Q}^2 \,.
$
\end{proof}
So the framework introduced in~\cite{verfuerth-2010} for an a posteriori analysis of a Backward Euler discretization of Stokes problem still applies here with Chorin-Temam scheme
(it applies with any scheme provided the reconstructions $\bu^{\dt}$, $p^{\dt}$ are defined using appropriate discrete variables).
Though, the point is now to let not only the residuals, but also the two last terms in~\eqref{eq:error_ct_104}, be easily and sharply estimated
(contrary to the fully discrete Backward Euler case in~\cite{verfuerth-2010}, these terms cannot be neglected here because they can be of the same order as the error).
We draw the following conclusions.
First, Prop.~\ref{above} and~\ref{below} suggest that the procedure of~\cite{verfuerth-2010} should be 
modified here 
to estimate the error 
\begin{equation}
 \label{error}
\mynu\|\grad e_u\|_{L^2(0,T;\Q^{d\times d})}^2 +\|\partial_t e_u + \grad e_p\|_{L^2(0,T;\U')}^2 
\end{equation}
a posteriori in a more robust way than by the 
estimator~\eqref{estimator1} obtained straightforwardly from the estimations above:
\begin{equation}
\label{estimator1} 
 \mynu\|\grad\bu^{\dt,+}-\grad\bu^{\dt}\|_{L^2(0,T;\Q^{d\times d})}^2 + \mynu\|\div \bu^{\dt}\|_{L^{2}(0,t;{\Q})}^2 + \|\div\partial_t\bu^{\dt}\|_{L^{1}(0,T;{\Q})}^2 \,.
\end{equation}
For instance, if one replaces~\eqref{sharp1} with the 
following upper bound~\eqref{notsharp1}, on noting~\eqref{eq:constants} and~\eqref{eq:H1stability},
\begin{equation}
\label{notsharp1}
\int_0^t \left| ( e_u, \partial_t\Pi e_u ) \right| 
\lesssim \|\grad e_u\|_{L^2(0,t;{\Q}^d)} \|\div\partial_t\bu^{\dt}\|_{L^{2}(0,t;\Q)} \,,
\end{equation}
then bounds similar to~\eqref{eq:up} and~\eqref{eq:down} hold but with $\|\div\partial_t\bu^{\dt}\|_{L^{2}(0,t;\Q)}$ instead of $\|\div\partial_t\bu^{\dt}\|_{L^{1}(0,t;\Q)}$
and {\it without invoking discretization parameters} like $N$ and $\dt^n$, 
which suggests the a posteriori error estimator~\eqref{estimator2} more robust than~\eqref{estimator1}:
\begin{equation}
\label{estimator2} 
\mynu\|\grad\bu^{\dt,+}-\grad\bu^{\dt}\|_{L^2(0,T;\Q^{d\times d})}^2 + \mynu\|\div \bu^{\dt}\|_{L^{2}(0,T;{\Q})}^2 + \|\div\partial_t\bu^{\dt}\|_{L^{2}(0,T;{\Q})}^2\,.
\end{equation}
Of course, this is not a fully efficient estimator yet, since it is a priori not bounded above {\it and} below by the error~\eqref{error}, 
even if one neglects the source error $\|\f-\f^{\dt}\|_{L^2(0,T;\Q^d)}^2$ of ``high'' order $O(\dt^2)$ -- recall~\eqref{apriori} --.
It is nevertheless useful in some cases, as shown in the next section. Second,~\eqref{estimator2} sometimes improves some estimators in the literature like
\begin{equation}
\label{estimator3} 
 \mynu\|\grad\bu^{\dt,+}-\grad\bu^{\dt}\|_{L^2(0,T;\Q^{d\times d})}^2 + \sum_{n=0}^{N-1}\|\dt^{n+1}\grad p^{n+1}-\dt^n\grad p^n\|_{\Q}^2
\end{equation}
that was proposed in~\cite{kharrat-mghazli-2010,kharrat-mghazli-2011}.
Clearly, for small $\dt$, our estimator is larger than the one proposed in~\cite{kharrat-mghazli-2010,kharrat-mghazli-2011}, 
on noting 
$$
\|\dt^{n+1}\grad p^{n+1}-\dt^n\grad p^n\|_{\Q}^2 \lesssim \|\div \bu^{n+1/2}-\div \bu^{n-1/2}\|_{\Q}^2 \lesssim \dt \left( \dt^n \|\div \partial_t \bu^{\dt}\|_{\Q}^2(t) \right)
$$
for $t\in(t_n,t_{n+1}]$, after using a Poincar\'e inequality with~\eqref{eq:chorin-temamPoisson2C}.
%
And, the numerical example of the following Section~\ref{num} indeed shows that~\eqref{estimator2} is a better upper-bound than~\eqref{estimator3},
at least when the error 
is not mainly driven by $\|\div \bu^{\dt}\|_{L^{\infty}(0,T;{\Q})}^2$. 

\section{Numerical results}
\label{num}

%

We want to bring numerical evidences that 
estimator~\eqref{estimator2} is sometimes i) useful 
and ii) better than~\eqref{estimator1} and~\eqref{estimator3}. 
Given $\lambda>0$, we numerically compute the efficiencies of the three estimators 
using discrete approximations of
$$ \bu=\pi\sin(\lambda t)\left(\sin(2\pi y)\sin(\pi x)^2;-\sin(2\pi x)\sin(\pi y)^2\right) \qquad p=\sin(\lambda t)\cos(\pi x)\sin(\pi y) $$
in 
$\D\equiv(-1,1)\times(-1,1)$ ($d=2$), 
with $t\in(0,T)$ uniformly discretized by time steps $\dt=T/N$ ($N\in\mathbb{N}$) when $\mynu=1$.

We discretize in space the velocity components and the pressure with, respectively, continuous $\mathbb{P}_2$ and $\mathbb{P}_1$ Finite-Elements functions,
i.e. in conforming discrete spaces $\U_h\subset\U,\Qh\subset(\Q\cap H^1(\D))$ defined on regular simplicial meshes of $\D$. 
The a posteriori analysis of Section~\ref{apost} still applies with right-hand side in~\eqref{eq:error_ct_104} defined using now fully-discrete approximations. 
Then indeed, following~\cite{verfuerth-2010}, one can 
decompose the fully-discrete residuals in a 
sum of two terms, one accounting for space-discretization errors and one for time-discretization errors. 
The two last terms in the (new) right-hand side of~\eqref{eq:error_ct_104} remain the same (they are explicitly computable).
This yields estimators linked to the time discretization which are exactly the space-discrete counterparts of the terms in the bounds~\eqref{eq:up} and~\eqref{eq:down}.
Moreover, in our numerical example, space discretization errors 
prove negligible in comparison with time 
discretization errors (as already observed in~\cite{guermond-minev-shen-2006} for $\lambda=1$).
We thus next show only numerical results obtained for one sufficiently fine mesh (with more than $10^5$ vertices). 

\begin{figure}
 \includegraphics[scale=.4]{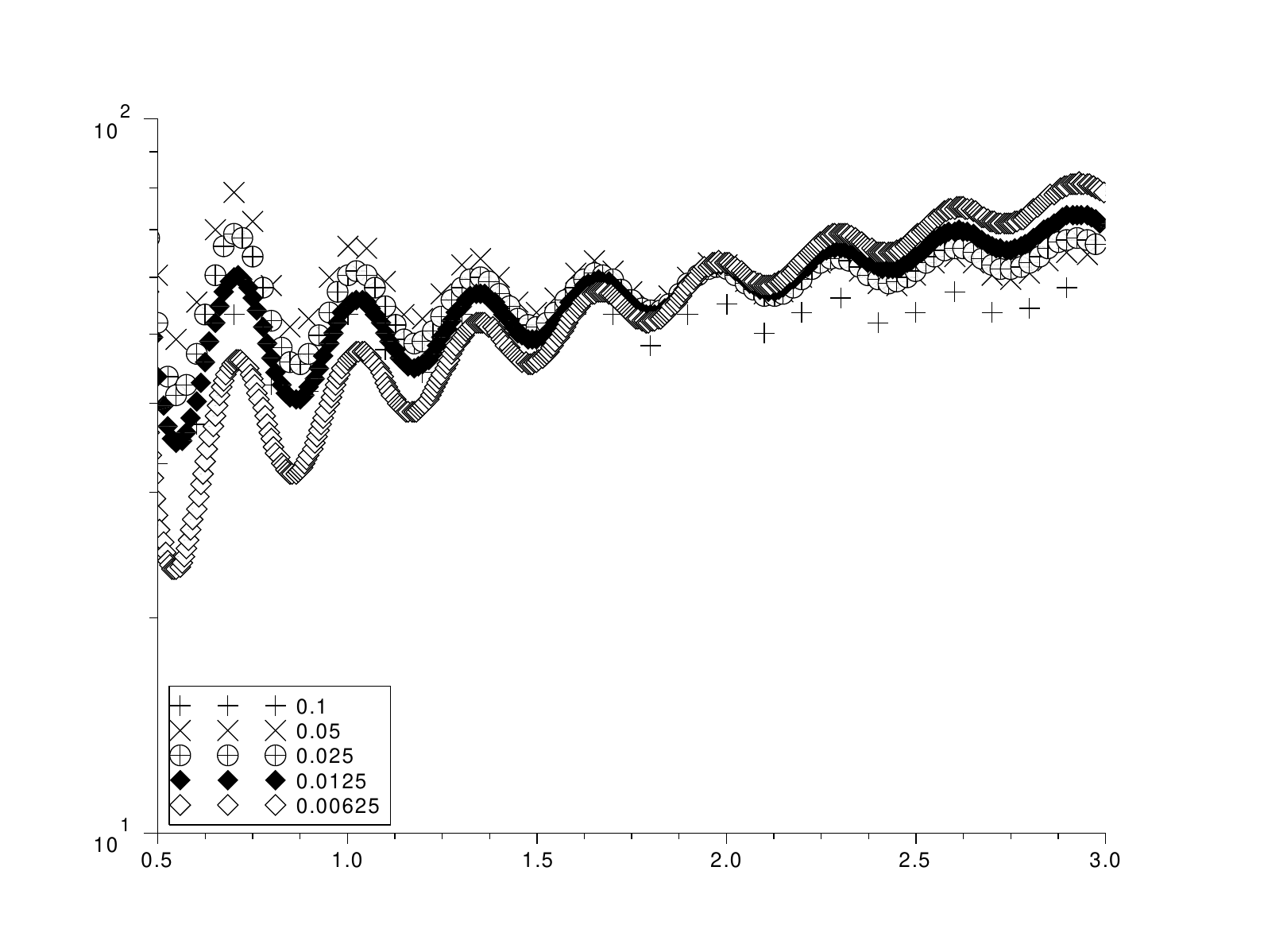}
 \includegraphics[scale=.4]{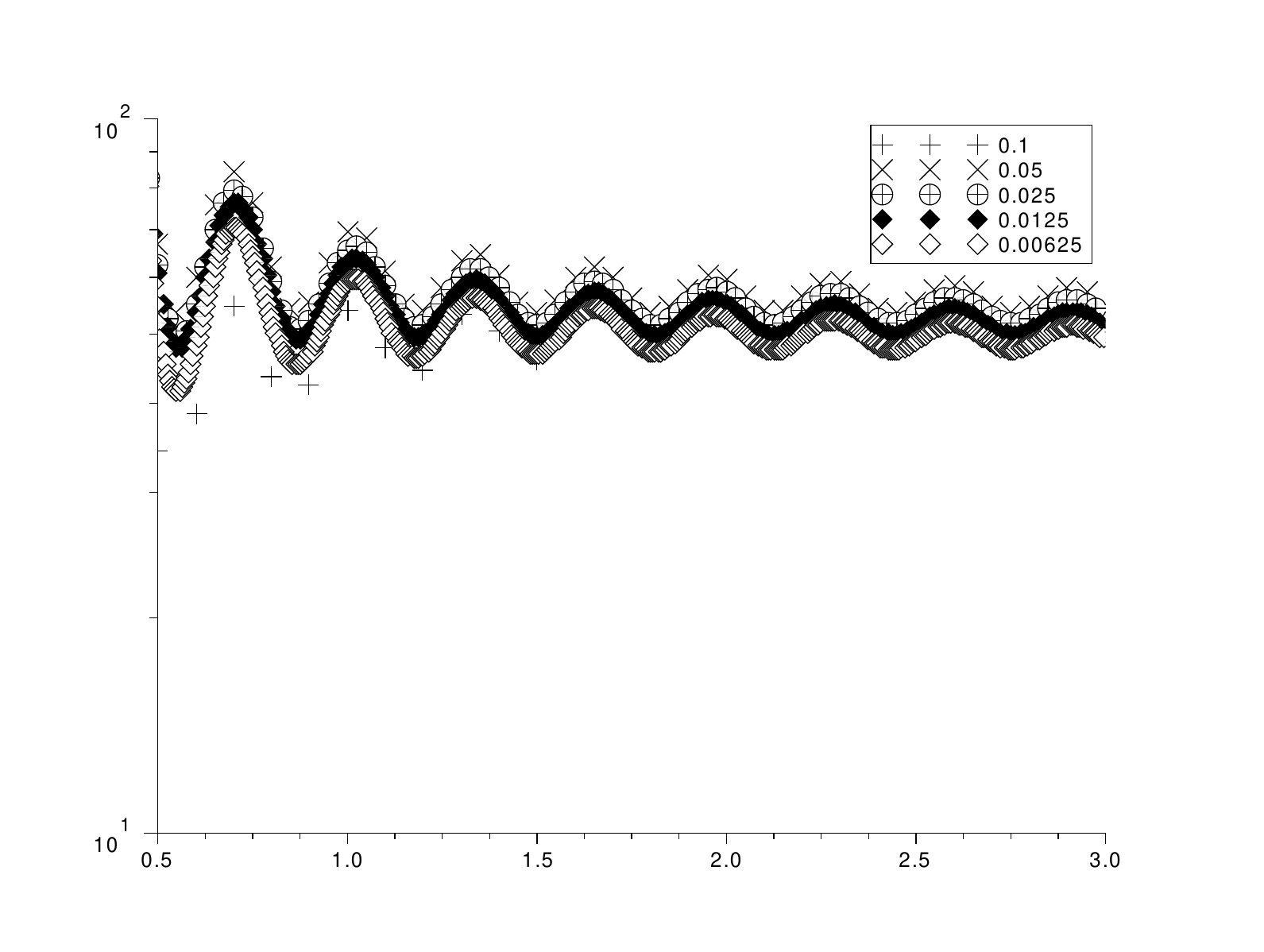}
\\
 \includegraphics[scale=.4]{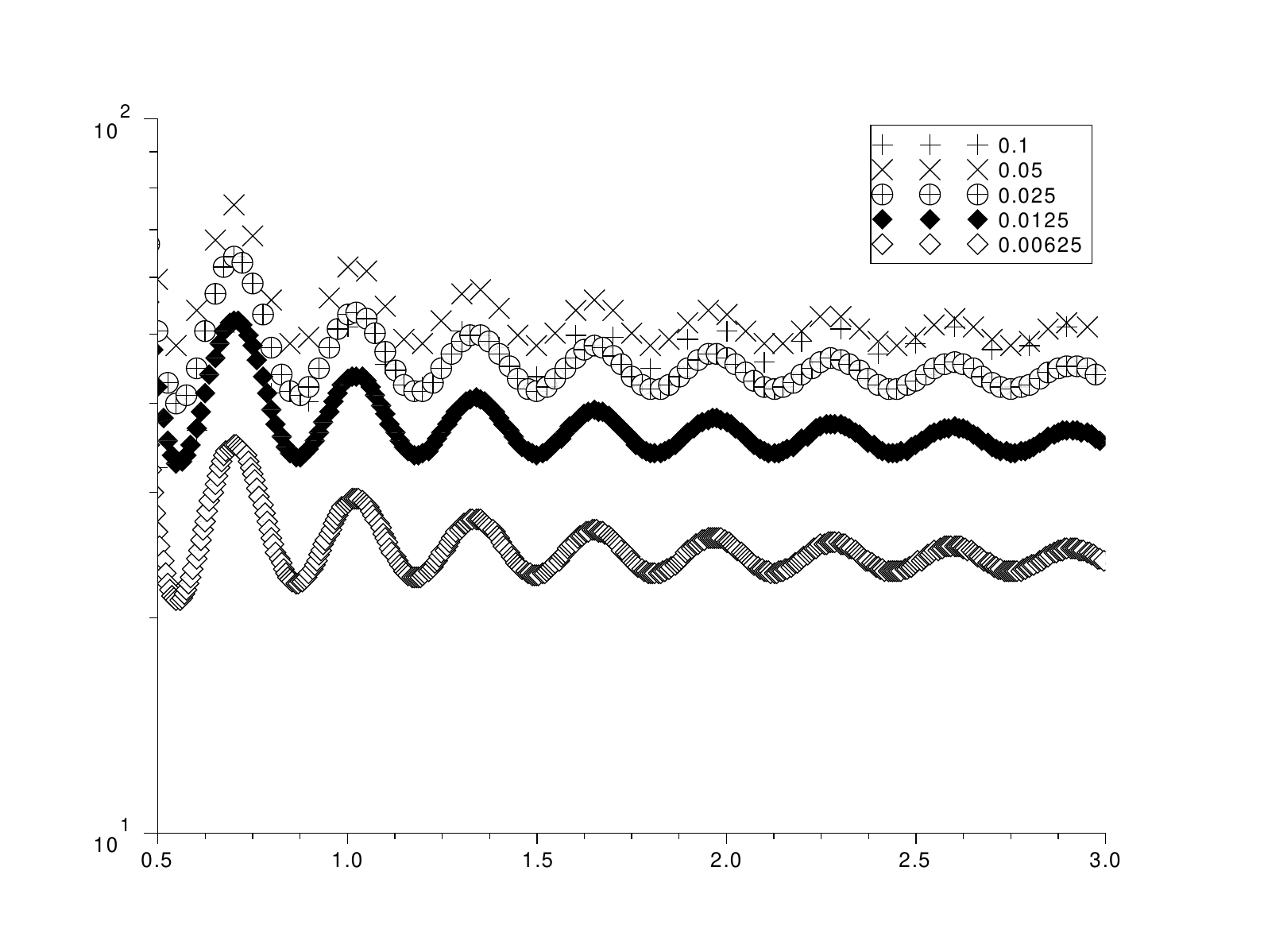}
 \includegraphics[scale=.4]{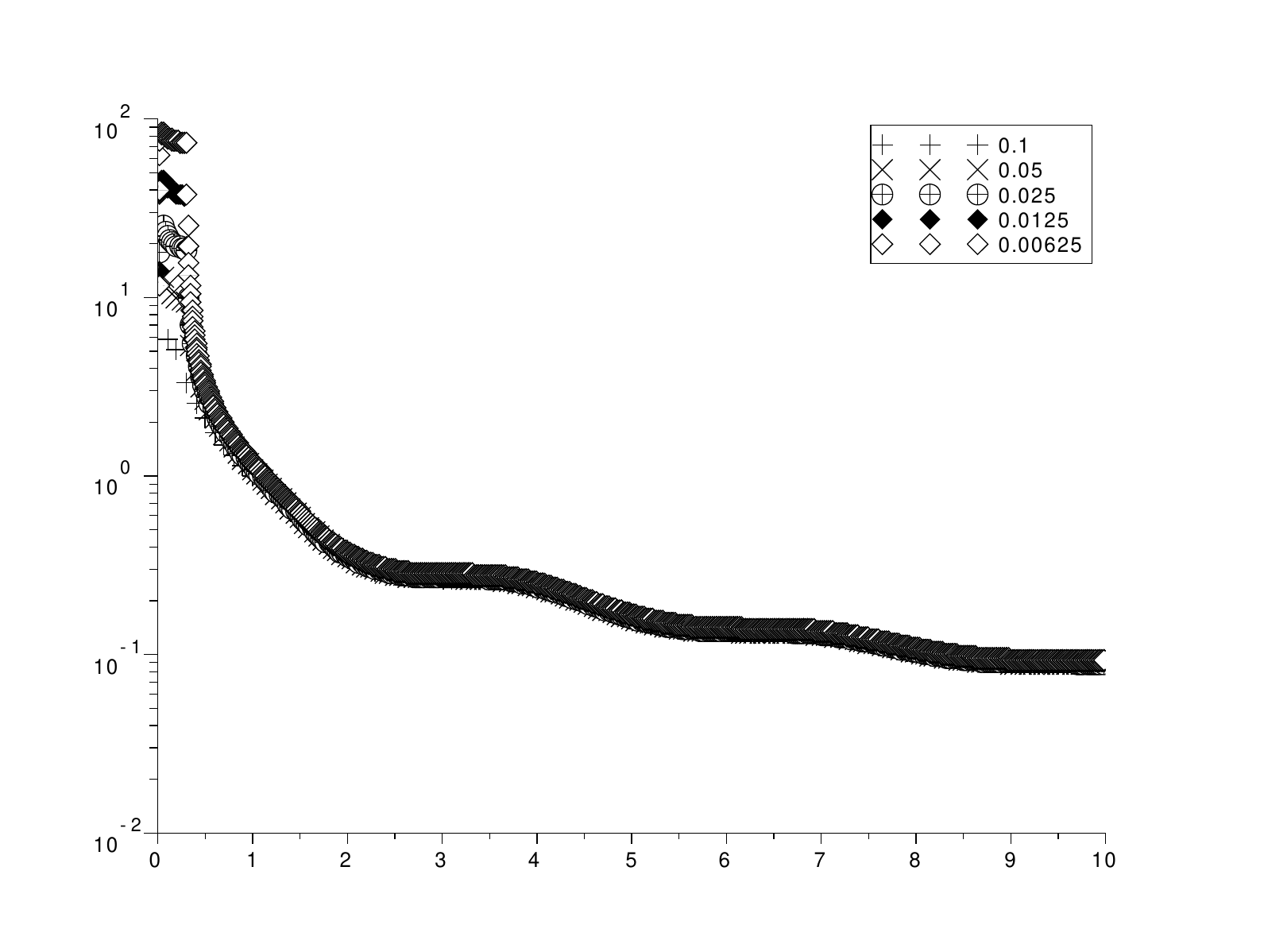}
\caption{\label{fig}
For $\dt=.1,.05,.025,.0125,.00625$, effectivities in log scale (as a function of $T$) of~\eqref{estimator1} -- top left --,~\eqref{estimator2} -- top right --, and~\eqref{estimator3} -- bottom left -- ($\|\div \bu^{\dt}\|_{L^{\infty}(0,T;{\Q})}^2$ included) at estimating~\eqref{error} when $\lambda=10,T\le3$; and $\|\div \bu^{\dt}\|_{L^{\infty}(0,T;{\Q})}^2/$ error~\eqref{error} -- bottom right --  when $\lambda=1,T\le10$.
}
\end{figure}

We compare the effectivities of (space-discrete versions of) the a posteriori error estimators~\eqref{estimator1},~\eqref{estimator2} and~\eqref{estimator3} evoked in the previous section for the (space-discrete) error $\|\grad e_{u_h}\|_{L^2(0,T;\Q^{d\times d})}^2+\|\partial_t e_{u_h} + \grad e_{p_h}\|_{L^2(0,T;\U')}^2$.
One clearly sees from the numerical results obtained for $\lambda=10,T\le3$ in Fig.~\ref{fig} that 
i)~\eqref{estimator1} is not robust when $\dt$ is too small or $T$ too large compared with~\eqref{estimator2}, and  
ii)~\eqref{estimator2} is better than the estimator~\eqref{estimator3} in so far as, for that specific case, it has the same decay rate than the error~\eqref{error} and not a faster one like~\eqref{estimator3}.
Though, our estimator~\eqref{estimator2} is 
still not fully efficient, even when adding the term $\|\div \bu^{\dt}\|_{L^{\infty}(0,T;{\Q})}^2$ to~\eqref{estimator2}, since it is not bounded above {\it and} below by the error.
Furthermore, if we use it as such (that is as a sum of terms without coefficients), 
in some cases, it also fails (like~\eqref{estimator3}) at evaluating correctly the error.
For instance when $\lambda=1$, the error~\eqref{error} scales like $\|\div \bu^{\dt}\|_{L^{\infty}(0,T;{\Q})}^2$ with respect to $\dt$, 
while the other terms in~\eqref{estimator2} are of higher-order in $\dt$.
But this cannot be observed unless $\dt$ is very small,
even if we use~\eqref{estimator2} 
plus $\|\div \bu^{\dt}\|_{L^{\infty}(0,T;{\Q})}^2$ as an estimator\footnote{
Note that in fact we also added the term $\|\div \bu^{\dt}\|_{L^{\infty}(0,T;{\Q})}^2$ to~\eqref{estimator1},~\eqref{estimator2} and~\eqref{estimator3} in Fig.~\ref{fig}, 
but it is small compared to other terms, 
thus unseen. 
}
insofar as the magnitude of the latter term is much smaller than the former ($10^{-1}$ vs. $10^2$).
Then, for too large $\dt$, the effectivity of our estimator also decays, and the error still cannot be evaluated confidently. 
So, without even mentionning the error $\|e_u\|_{L^\infty(0,T;\Q^d)}$,
the question how to estimate a posteriori error discretizations in Chorin-Temam scheme {\it efficiently and robustly} (in all cases) thus remains open.
One should at least coefficient adequately the terms in the estimator above~\eqref{estimator2} plus $\|\div \bu^{\dt}\|_{L^{\infty}(0,T;{\Q})}^2$.
We nevertheless hope to have shed new light on the problem.





\bibliographystyle{model1-num-names}



\end{document}